\documentclass[a4paper,10pt]{amsart}

\usepackage[english]{babel}
\usepackage{dsfont} 
\usepackage{graphicx}
\usepackage{color}
\usepackage{enumerate}
\usepackage{array,multirow}

\allowdisplaybreaks

\usepackage{amssymb} 
\usepackage{amsthm} 
\usepackage{dsfont}

\newtheorem{thm}{Theorem}

\newtheorem{lem}[thm]{Lemma}

\newtheorem{prop}[thm]{Proposition}

\theoremstyle{definition}

\newtheorem{exm}[thm]{Example}
\newtheorem{rem}[thm]{Remark}

\newcommand{\dT}{\mathds{T}}
\newcommand{\dN}{\mathds{N}}
\newcommand{\dZ}{\mathds{Z}}

\newcommand{\dR}{\mathds{R}}
\newcommand{\dC}{\mathds{C}}
\newcommand{\ii}{\mathrm{i}}

\newcommand{\cI}{\mathcal{I}\,}
\newcommand{\cJ}{\mathcal{J}\,}

\newcommand{\her}{\mathrm{her}\,}

\newcommand{\ov}{\overline}

\newcommand{\cF}{\mathcal{F}}

\newcommand{\cZ}{\mathcal{Z}}

\newcommand{\sN}{\mathsf{N}}

\newcommand{\fg}{\mathfrak{g}}

\author{Claus Scheiderer}
\address{University of Konstanz,  Fachbereich Mathematik und Statistik, D-78457
Konstanz, Germany}
\email{claus.scheiderer@uni-konstanz.de}

\author{Konrad Schm\"udgen}
\address{University of Leipzig, Mathematical Institute, Augustusplatz 10/11, D-04109 Leipzig, Germany}
\email{\tt schmuedgen@math.uni-leipzig.de}

\date{}
\begin{document}

\begin{title}{Moment Property and Positivity for some Algebras of Fractions}
\end{title}
\date{\today}
\dedicatory{Dedicated to the memory of Alexander Prestel (17.1.1941--4.3.2024) }

\begin{abstract}
T. M. Bisgaard \cite{bisgaard} proved that the  $*$-algebra
$\dC[z,\ov{z},z^{-1},\ov{z}^{-1}]$\, has the moment property, that is,
each positive linear functional on this $*$-algebra is a moment
functional. We generalize this result to polynomials in $d$ variables
$z_1,\dots,z_d.$ We prove that there exist $3d-2$ linear polynomials
as denominators such that the corresponding $*$-algebra has the
moment property, while  for $3$ linear polynomials in case $d=2$
the moment property always fails. Further, it is shown that for
the real algebras $\dR[x,y,\frac1{x^2+y^2}]$ (the hermitean
part of $\dC[z,\ov{z},z^{-1},\ov{z}^{-1}]$) and
$\dR[x,y,\frac{x^2}{x^2+y^2},\frac{xy}{x^2+y^2}]$, all positive
semidefinite elements are sums of squares.  These results are used
to prove that for the semigroup $*$-algebras of  $\dZ^2$,
$\dN_0\times \dZ$ and ${\sN}_+:=\{(k,n)\in \dZ^2:k+n\geq 0\}$, all
positive semidefinite elements are sums of hermitean squares.
\end{abstract}
\maketitle

14 P 10 (Secondary).

\textbf{Key  words:}  complex moment problem, fibre theorem, sums of squares, $*$-semigroup

\section{Introduction}

The starting point for this paper is the following remarkable theorem
of  T. M. Bisgaard \cite{bisgaard}: Each positive semidefinite sequence
on the $*$-semigroup $\dZ^2$, with involution $(m,n)^*=(n,m)$ for
$n,m\in \dZ$, is  a moment sequence.
It is well known that the map $(m,n)\mapsto z^m\ov{z}^n$\, gives a $*$-isomorphism
from the semigroup $*$-algebra $\dC[\dZ^2]$ to the $*$-algebra\,
$\dC[z,\ov{z},z^{-1},\ov{z}^{-1}]$\, of  complex Laurent polynomials in $z$
and $\ov{z}$, with involution $z_j^*:=\ov{z_j}$. Using this $*$-isomorphism,
Bisgaard's theorem can be reformulated by saying  that each positive linear functional $L$ on the $*$-algebra $\dC[z,\ov{z},z^{-1},\ov{z}^{-1}]$\, is a moment functional. That is, if $L(f^*f)\geq 0$ for all $f\in \dC[z,\ov{z},z^{-1},\ov{z}^{-1}]$,
there exists  a Radon measure $\mu$ on $\dC\backslash \{0\}$  such that
\begin{align*}
L(p)=\int_{\dC\backslash \{0\}} p(z,\ov{z})\, d\mu(z)\quad \text{\textrm{for}} ~~~p\in\dC[z,\ov{z},z^{-1},\ov{z}^{-1}]\,.
\end{align*}
In this paper we are looking for generalizations of this important result  to higher dimensions.

Throughout we let $d\in \dN$ and  consider  the complex unital
$*$-algebra\, $$\dC_d[z,\ov{z}]:=\dC[z_1,\dots,z_d,\ov{z_1},\dots, \ov{z_d}]$$
of  complex  polynomials in $z_1,\dots,z_d$ and $\ov{z_1},\dots,\ov{z_d}$.
The involution $p\to p^*$ of $\dC_d[z,\ov{z}]$ is uniquely determined by
the requirement $z_j^*=\ov{z_j}$ for $j=1,\dots,d$.

Suppose that $\cF:=\{f_1,\dots,f_m\}$ is a finite subset of
$\dC_d[z,\ov{z}]$. To avoid trivial cases we assume throughout that none
of the polynomials $f_j$ is constant. Let
\begin{align}\label{defA}
A:=\dC\big[z_1,\dots,z_d,\ov{z_1},\dots,\ov{z_d},f_1^{-1},\dots,f_m^{-1},(\,\ov{f_1}\, )^{-1},\dots,(\, \ov{f_m}\,)^{-1}\big].
\end{align}
 Clearly,  $A$ is the algebra generated by $\dC_d[z,\ov{z}]$ and $f^{-1}$, $(\, \ov{f}\, )^{-1}$ where $f:=f_1f_2\cdots f_m$.
 Then $A$ is a complex unital $*$-algebra which contains $\dC_d[z,\ov{z}]$ as a $*$-subalgebra. Observe that Bisgaard's theorem corresponds to the case  $d=1, m=1, \cF=\{z\}$.

The first problem we study is when the $*$-algebra $A$ has the moment property:\smallskip

{\it 1.) When is each positive linear functional on $A$  a moment functional?}\smallskip

\noindent
For this we use the fibre theorem which was proved in \cite{sch03} in the polynomial case and in \cite{sch16} for arbitrary finitely generated real algebras. Our method of proof is essentially based on real algebra, while the problems and the  results   are formulated for commutative complex $*$-algebras.

Our main result concerning the first problem says that there are $m:=3d-2$\linebreak
linear polynomials $f_j$  in $z_1,\dots,z_d$ such that $A$ satisfies the
moment property (Theorem \ref{mainthmp}). If $d=2$ we show, for arbitrary
linear polynomials $f_1,f_2,f_3$, that the corresponding $*$-algebra does
not have the moment property (Theorem \ref{mpfails}).

\smallskip
Now let $A$ be an arbitrary finitely generated  commutative unital
complex $*$-algebra.  Its hermitean part $B:=A_\her$ is a finitely
generated commutative real algebra. We denote by $\Sigma B^2$  the set of all sums of squares in $B$ and by $\Sigma_h A^2$ the set of all sums of hermitean squares $a^*a$ in $A$. Further, $B_+$ and $A_+$ are the sets of elements of $B$ and $A$, respectively, that are nonnegative on all characters. Then $A$ has the moment
property if and only if $B$ does, or equivalently, $B_+$ is the
{\it closure} of $\Sigma B^2$ in the finest locally convex topology (see \cite[Proposition 13.9(ii)]{sch17}). In particular, if $\Sigma B^2=B_+$,  then $A$ and $B$ have the moment property. The converse is not true in general and the requirement $\Sigma B^2=B_+$ is stronger than the moment property. A second problem is the following:\smallskip

{\it  2.) If $A$ has the moment property,  when does it follow
that $\Sigma B^2=B_+$ (or equivalently, $\Sigma_h A^2=A_+$)?}\smallskip

For the real algebras
\begin{align}\label{realfraction}
 \dR\Big[x,y,\frac1{x^2+y^2}\Big]~~~ {\rm and}~~~\dR\big[x,y,\frac{x^2}{x^2+y^2},\frac{xy}{x^2+y^2}\Big]
\end{align} we prove that all positive semidefinite elements are sums of squares (Theorems \ref{bisg} and \ref{bisgbdd}). As shown in \cite{sw}, they same is true for the real
algebra $\dR[x,y,z]/(1-x^2-y^2)$ (Theorem \ref{cylthm}).
It is known for each of the three $*$-semigroups
\begin{align}\label{defS}
S=\dZ^2, ~ \dN_0\times \dZ, ~ {\sN}_+:=\{(k,n)\in \dZ^2:k+n\geq 0\}
\end{align}
that
  all  positive semidefinite sequences are moment sequences, or equivalently,
that
  the semigroup $*$-algebra $\dC[S]$  has the moment property. Using
  Theorems  \ref{bisg}--\ref{bisgbdd} we prove the stronger result that
  $\Sigma_h\dC[S]^2 = \dC[S]_+$ (Theorem \ref{semigroup}).

This paper is organized as follows.
In Section \ref{prel} we collect basic definitions and preliminary results. In particular, we recall the special case of the fibre theorem that we shall use. We discuss in great detail the interplay between properties of a commutative complex $*$-algebra $A$ and its hermitean part $A_\her$.

In Section \ref{algrational} we study the algebra $A$ defined by (\ref{defA}), and derive a convenient parametrization of the non-empty fibre sets. This will be the technical tool for applying the fibre theorem.

In  Section \ref{genbisgaard} we investigate the first problem
in the case when $f_1,\dots,f_m$ are linear polynomials in
$z_1,\dots,z_d$.  The main result of this section  can be considered
as a generalization of Bisgaard's theorem to higher dimensions.
Also, we  derive a result where  the moment property fails.

Sections \ref{psdsos} and \ref{semigroupappl} deal with the second
problem. In Section \ref{psdsos} we show that positive semidefinite
elements are sums of squares for the real algebras (\ref{realfraction}), while in Section \ref{semigroupappl}  it is shown
 that equality $\Sigma_h\dC[S]^2=\dC[S]_+$ holds for the three $*$-semigroups $S$ in (\ref{defS}).

\section{Preliminaries}\label{prel}

Suppose $B$ is a \emph{finitely generated real unital algebra}.

We denote by $\Sigma B^2$ the {\it sos cone} of $B$, that is,
$\Sigma B^2$ is the set of all finite sums of squares $\sum_ib_i^2$
with $b_i\in B$. A character of $B$ is an algebra homomorphism  $\chi:B\to \dR$ satisfying $\chi(1)=1$.  The set of all characters of $B$ is denoted by $\widehat{B}$. Further, define
\begin{align*}
B_+=\{ b\in B: \chi(b)\geq 0 ~~ \textrm{for all}~~\chi\in \hat{B}\, \}.
\end{align*}
The elements of $B_+$ are called {\it positive semidefinite},  or briefly {\it psd}. Obviously, $\Sigma B^2\subseteq B_+$.

Let us fix a set $\{b_1,\dots,b_d\}$  of generators of the algebra $B$.
Then  there exists a unique surjective  unital algebra homomorphism
$\pi: \dR[x_1,\dots,x_d]\to B$ such that $\pi(x_j)=b_j$ ($j=1,\dots,d$).
The kernel $J$  of $\pi$ is an ideal of $ \dR[x_1,\dots,x_d]$ and $B$ is
isomorphic to the quotient algebra\,  $\dR[x_1,\dots,x_d]/J $.
Each character $\chi$ of $B$ is  uniquely determined by  the point
$x_\chi:=(\chi(b_1),\dots,\chi(b_d))$ of $\dR^d$.  We identify
$\chi$ with $x_\chi$ and write $f(x_\chi):=\chi(f)$ for $f\in B$.
Under this identification,\, $\widehat{B}$ becomes the real algebraic set
$\widehat{B}=\cZ(J):=\{x\in \dR^d:p(x)=0, ~p\in J\}.$
Since $\cZ(\cJ)$ is closed in $\dR^d$, $\widehat{B}$ is a locally compact Hausdorff space in the induced topology of $\dR^d$ and elements of $B$ can be considered as continuous functions on $\widehat{B}.$

 A linear functional $L:B\to \dR$  is called a {\it moment functional} if there exists a (positive) Radon measure $\mu$ on the locally compact space $\widehat{B}$ such that for all $f\in B$,
 \begin{align}\label{lrepmu}
 L(f)=\int_{\widehat{B}} f(x)\, d\mu(x).
 \end{align}
 If $L$ is a moment functional on $B$, then obviously $L(f^2)=\int_{\widehat{B}} f^2\, d\mu\geq 0$ for all $f\in B$, so $L$ is non-negative on $\Sigma B^2$.  The converse implication does not hold in general, that is, being nonnegative on $\Sigma B^2$  does not imply that $L$ is a moment functional.  It holds for $B=\dR[x]$ by Hamburger's solution of the moment problem, but not for $B=\dR[x_1,\dots,x_d]$ if $d\geq 2$. Finding  algebras for which the converse  is true is an important and difficult task of the moment problem theory.

 We shall say that $B$ has the {\it moment property} if each linear functional $L:B\to \dR$ such that $L(f^2)\geq 0$ for all $f\in B$ is a moment functional.

Next we turn to the fibre theorem.
Suppose
 $(h_1,\dots,h_n)$\, is an $n$-tuple of  elements  $h_j\in B$ which are
 {\it bounded}  on the character set $\widehat{B}$. For $\lambda=(\lambda_1,\dots,\lambda_n)\in \dR^n $, let $\cI(\lambda)$ denote the ideal of $B$ generated by the elements $h_1-\lambda_1,\dots,h_n-\lambda_n$. The ideal $\cI(\lambda) $ is called the {\it fibre ideal} of $\lambda$. Its zero set is the {\it fibre set}
\begin{align*}
\cZ(\lambda)=\{ x\in \widehat{B}: h_1(x)=\lambda_1,\dots,h_n(x)=\lambda_n \}.
\end{align*}
It is possible that the fibre set $\cZ(\lambda)$ is empty. Let $\widehat{\cI}(\lambda)$ be the ideal of all  $a\in B$ which vanish on the set $\cZ(\lambda)$. This means that $\widehat{\cI}(\lambda)$ is the real radical of $\cI(\lambda)$.

In this paper we will use the following special case of the fibre theorem \cite{sch03}, \cite{sch16}. It is the equivalence $(i)\leftrightarrow (iii)'$ of  \cite[Theorem 13.10]{sch17}  in the special case when the  preordering  of $B$ is the set of all sums of squares.
\begin{prop}\label{fibreth1}
Under the preceding assumptions  the following are equivalent:
\begin{itemize}
\item[\em (i)]  $B$ has the moment  property.
\item[\em (ii)]    For each $\lambda\in \dR^n$ such that the fibre set $\cZ(\lambda)$ is not empty, the quotient algebra $B /\widehat{\cI}(\lambda)$ has the moment property.
\end{itemize}
\end{prop}
Let us retain the preceding setup and consider an important special case which will be needed later.
\begin{prop}\label{mplines}
Suppose that $p_1,\dots,p_k,q_1,\dots,q_k$ are  polynomials of $\dR[x_1,\dots,x_n]$ such  that $p_j\neq 0$ or $q_j\neq 0$ for each $j=1,\dots,k$.
Let
$$B:=\dR[x_1,\dots,x_n,(p_1^2+q_1^2)^{-1},\dots,(p_k^2+q_k^2)^{-1}]$$ If all non-empty fibre sets  of $B$ are subsets of straight lines on $\dR^n$, then $B$ has the moment property.
\end{prop}
\begin{proof}
Let $\cZ$ be a fixed non-empty fibre set of some fibre ideal
$\cI$ of $B$. By assumption, $\cZ$ is contained in some line,
that is,   there exist $a\in \dR^n$, $a\neq 0$, and $y=(y_1,\dots,y_n)\in \dR^n$
such that $\cZ\subseteq \{y+ta: t\in \dR\}$. For simplicity assume that $a_1\neq 0$. If $\cZ$ consists only of $y$, then all  $x_j$ are constant on $\cZ$. Hence $B/ \widehat{\cI} = \dR$ which obviously has the moment property. Suppose now that $\cZ\neq \{y\}$ and let  $x\in \cZ$, $x\neq y$. Then  $x=y+ta$ with $t\neq 0$, so that $x_j=y_j+a_j a_1^{-1} (x_1-y_1)$. Hence the quotient algebra $C:=B/\widehat{\cI}$ is generated by $x_1$ and  $(\tilde{p}_j(x_1)^2+\tilde{q}_j(x_1))^2)^{-1}$  for certain polynomials $\tilde{p}_j, \tilde{q}_j\in \dR[x_1]$, $j=1,\dots,k$.

The character space of $C$ consists of the point evaluations  at the set $$\widehat{C}=\{t\in \dR:\tilde{p}_j(t)^2+\tilde{q}_j(t)^2\neq 0,  j=1,\dots,k\}.$$ Let $f\in C$ be such that $f(t)\geq 0$ for all $t\in \widehat{C}$. From the definition of $C$ it follows that for sufficiently large $n\in \dN$,
$$
\tilde{f}(x_1):=f(x_1)[(\tilde{p}_1(x_1)^2+\tilde{q}_1(x_1)^2)\cdots(\tilde{p}_k(x_1)^2+\tilde{q}_k(x_1)^2)]^{2m}\in \dR[x_1].$$
Obviously, $\tilde{f}$ is also nonnegative on $\widehat{C}$ and hence on $\dR$.  Therefore, $\tilde{f}\in \sum \dR[x_1]^2$ which implies that
$$
f(x_1)=\tilde{f}(x_1)(\tilde{p}_1(x_1)^2+\tilde{q}_1(x_1)^2)^{-2m}\cdots(\tilde{p}_k(x_1)^2+\tilde{q}_k(x_1^2))^{-2m} \in \sum C^2.
$$

Therefore, it follows from Haviland's theorem \cite[Theorem 1.14]{sch17} that $C=B/\widehat{\cI}$  has the moment property. Since this holds for each fibre, $B$ itself has the moment property by
 Proposition \ref{fibreth1}.
 \end{proof}

Now we
suppose that $A$ is a \emph{finitely generated complex commutative  unital $*$-algebra}, with involution  denoted by $a\to a^*$. We recall some simple facts about $*$-algebras which can be found, for instance, in \cite{sch20}.

 The {\it hermitean part} $$B\equiv A_\her:=\{ a\in A:a^*=a\}$$ of $A$
 is a commutative real unital algebra and each element $a\in A$ can be
 uniquely written as $a=a_1+\ii a_2$, with $a_1,a_2\in B$. Let
 $\Sigma_h A^2$ denote the set of finite sums $\sum_ja_ja_j^*$ of
 hermitean squares, with $a_j\in A$.  Then $\Sigma_h A^2=\Sigma B^2$ (see e.g. \cite[Lemma 2.17]{sch17}).

A linear functional $L:A\to \dC$ is called {\it hermitean} if $L(a^*)=\overline{L(a)}$ for all $a\in A$.  Hermitean functionals on $A$ have real values on $B$. Hence linear functionals $L:B\to \dR$ are in one-to-one correspondence to hermitean linear functionals $L:A\to \dC$.

A linear functional $L:A\to \dC$ is said to be  {\it positive} if $L(a^*a)\geq 0$ for all $a\in A$. Each positive linear functional is hermitean.

 A character of $A$ is a unital $*$-homomorphism of $A$ into $\dC$. The set of characters of $A$ is denoted by $\widehat{A}$. Define \begin{align*}
A_+=\{ a\in A: \chi(a)\geq 0 ~~ \textrm{for all}~~\chi\in \hat{A}\, \}.
\end{align*}
 Characters of $A$ take real values on $B$. The restriction map $\chi\to \chi\lceil B$ is a bijection of $\hat{A}$ on $\hat{B}$ and it maps $A_+$ onto $B_+$.

 A linear functional $L:A\to \dC$ is called a {\it moment functional} if there exists a Radon measure $\mu$ on the locally compact space $\widehat{A}$ such that $L(f)=\int f d\mu$ for all $f\in A$.  Clearly, each moment functional on $A$ is positive. The restriction of a moment functional on $A$ is obviously a moment functional on $B$. Conversely, if we have a moment functional on $B$, then its unique extension to a hermitean functional on $A$ is a moment functional on $A$

We say that $A$ has the {\it moment property} if each positive linear functional $L$ on $A$ is a moment functional on $A$. Since $\Sigma_h A^2=\Sigma B^2$,  $A$ has the moment property if and only if its hermitean part $B$ does.

Summarizing, we have reduced all needed properties and facts on $A$ to
their counterparts on $B$.

The following is a version of {\it Haviland's theorem}. The proof follows from \cite[Theorem 1.14]{sch17} combined with the preceding discussion.
\begin{prop}\label{haviland} Suppose that $A$ is a finitely generated  complex unital $*$-algebra and $B:=A_\her$ denotes its hermitean part.
For a hermitean linear functional $L:A\to \dC$ the following are equivalent.\begin{itemize}
\item[\em (i)]  $L(a)\geq 0$ for $a\in A_+$.
\item[\em (ii)]$L(a)\geq 0$ for $a\in B_+$.
\item[\em (iii)] $L$ is a moment functional on $A$.
\item[\em (iv)]   $L\lceil B$ is a moment functional on $B$.
\end{itemize}
\end{prop}
The following simple lemma is used in the next section to describe the fibre sets.
Suppose that $M$ is a set of non-zero complex numbers.
We consider the map
\begin{align*}\Phi\colon M\to\dC,\quad\Phi(w):=\frac{w}{\ov{w}}
\end{align*}  of $M$ to the unit circle $\dT$. For $\theta\in \dT$  let
\begin{align}\label{mtheta}
M_\theta:=\{w\in M:\Phi(w)=\theta\}
\end{align}
denote the fibre set of $\Phi$ at $\theta$.
\begin{lem}\label{prelim}
Let $\theta\in \dT$. We write $ \theta=e^{\ii \, 2\varphi}$ with unique  $\varphi\in [0,\pi)$. Let $\fg_\theta$ denote the line in $\dC$ through the origin and the point $e^{\ii \varphi}$. Then
\begin{align}\label{equi1}
M_\theta =\fg_\theta\cap M=\{u+\ii v\in M: u,v\in \dR, u\sin \varphi=v\cos\varphi\}.
\end{align}
\end{lem}
\begin{proof} Let $w\in M$ and let $w=re^{\ii\, \psi}$  be the polar decomposition of $w$. Since  $\Phi(w) = e^{\ii\, 2\psi}$, we obtain
\begin{align*}M_\theta=\Phi^{-1}(\theta)=M\cap\>\{\pm re^{\ii \, \varphi}\colon r>0\}\>= M\cap
\{re^{\ii\, \varphi}\colon0\ne r\in\dR\}=M\cap \fg_\theta,
\end{align*}
which gives the first equality of  (\ref{equi1}). The second  follows from the fact $\fg_\theta$ has the equation $u\sin \varphi=v\cos \varphi$, where $u+\ii v\in \dC,u,v\in \dR$.
\end{proof}

\section{Algebras of complex rational functions}\label{algrational}
Let us begin with some notations.

We write $z_j=x_j+\ii\, y_j$ for $j=1,\dots,d$, where $x_j, y_j$ are
real variables and  identify $z=(z_1,\dots,z_d)\in\dC^d$ with $x=(x_1,\dots,x_d,y_1,\dots,y_d)\in\dR^{2d}$. In this manner, we  obtain an $*$-isomorphism of  $*$-algebras $\dC_d[z,\ov{z}]= \dC[z_1,\dots,z_d,\ov{z_1},\dots, \ov{z_d}]$ and $\dC_{2d}[x,y]:=\dC[x_1,\dots, x_d,y_1,\dots,y_d]$, with involution determined by $x_j^*=x_j, y_j^*=y_j$, $j=1,\dots, 2d$. For notational simplicity we will identify the $*$-algebras  $\dC_d[z,\ov{z}]$ and $\dC_{2d}[x,y]$ in the sequel.
The hermitean part of the complex $*$-algebra $\dC_{2d}[x,y]$ is the real algebra $\dR_{2d}[x,y]:=\dR[x_1,\dots,x_d,y_1,\dots,y_d]$.

As discussed in the introduction,
$\{f_1,\dots,f_m\}$ is a finite subset of
$\dC_d[z,\ov{z}]$ and $A$ is the $*$-algebra defined by (\ref{defA}).
The decomposition $f_j=a_j+\ii\, b_j$ of $f_j$,  with $a_j=a_j^*, b_j=b_j^*\in \dC_d[z,\ov{z}]$, gives
\begin{align}
f_j=a_j+\ii b_j,~~ \textrm{ where}~~a_j,b_j\in \dR_{2d}[x,y],~ j=1,\dots,m.
\end{align}
Note that $a_j$ and $b_j$ are real-valued on $\dR^{2d}$, because $a_j^*=a_j, b_j^*=b_j$. Further,
\begin{align}\label{fjinverse}
f_j^{-1}=\frac{a_j-{\ii} b_j}{a_j^2+b_j^2}~~~\textrm{and} ~~(f_j^{-1})^*=(f_j^*)^{-1}=\frac{a_j+{\ii} b_j}{a_j^2+b_j^2} .
\end{align}

The character space of $\dC_d[z,\ov{z}]$ consists of the point evaluations at the points of $\dC^d$. From  the definition (\ref{defA}) of the algebra $A$   it follows that  the character space $\widehat{A}$  is given by the point evaluations at the set
\begin{align*}
\widehat{A} = \{z\in \dC^d: f_j(z)\neq 0,~ j=1,\dots,m\}.
 \end{align*}
 Let $B:=A_\her$ denote the hermitean part of the complex $*$-algebra $A$.
The functions
\begin{align*}
\frac{1}{2}\big((f_j^{-1})^*{+}f_j^{-1}\big)=\frac{a_j}{a_j^2+b_j^2},~\frac{1}{2\ii}\big((f_j^{-1})^* {-}f_j^{-1}\big)=\frac{b_j}{a_j^2+b_j^2},~
f_j^{-1}(f_j^*)^{-1}=\frac{1}{a_j^2+b_j^2}
\end{align*}
belong to $B$.
Then\, $A$\, is the complex  unital $*$-algebra and $B$ is the real unital algebra generated by the hermitean  functions
\begin{align}\label{hermitgen}
x_1,\dots,x_d,y_1,\dots,y_d, ~ \frac{1}{a_1^2+b_1^2},\dots,~\frac{1}{a_m^2+b_m^2}.
\end{align}
Clearly, the characters of $B$ are the point evaluations at the points of
\begin{align*}
\widehat{B} = \{ (x,y)\in \dR^{2d}: a_j(x,y)\neq 0~~\textrm{or}~~ b_j(x,y)\neq 0, ~j=1,\dots,m\}.
\end{align*}
The functions\, $x_j, y_j,\frac{a_j}{a_j^2+b_j^2}, \frac{b_j}{a_j^2+b_j^2},\frac{1}{a_j^2+b_j^2}$\,  are in general unbounded  on  $\widehat{B}$.
But the $2m$  functions
\begin{align*}
g_j(x,y):=\frac{a_j(x,y)^2- b_j(x,y)^2}{a_j(x,y)^2+b_j(x,y)^2}\, ,~ h_j(x,y): =\frac{2a_j(x,y)b_j(x,y)}{a_j(x,y)^2+b_j(x,y)^2}, ~~ j=1,\dots,m,
\end{align*}
belong to $B$ and they are {\it bounded} on the set  $\widehat{B}$. The values of the functions $g_j, h_j$ on $\widehat{B}$ are contained in the interval $[-1,1]$ and   we have
\begin{align}\label{unitcircle}
g_j(x,y)^2+h_j(x,y)^2=1\quad \textrm{for}~~ (x,y)\in \widehat{B}.
\end{align}

Our aim is to apply the fibre theorem to the algebra $B=A_\her$ and the bounded  functions $g_1,\dots,g_m,h_1,\dots,h_m$.

Let us fix a $2m$-tuple $\lambda=(\lambda_1,\dots,\lambda_{2m})\in \dR^{2m}$.
We denote by $\cI(\lambda)$  the ideal in $B$ generated by  the elements
\begin{align*}
g_j-\lambda_j ~~\textrm{and} ~~ h_j-\lambda_{m+j}, \quad j=1,\dots,m.
\end{align*}
and by $\cZ(\cI_\lambda$ the fibre set
\begin{align}\label{zlambda}
\cZ(\cI(\lambda))
=\{\,  (x,y)\in \widehat{B}: g_j(x,y)=\lambda_j,  ~h_j(x,y)=\lambda_{m+j} ~~\textrm{for}~ j=1,\dots,m\}.
\end{align}

Suppose that the fibre set $\cZ(\cI(\lambda))$ is not empty. Let $(x,y)$ be a point of $ \cZ(\cI(\lambda))$. Then, by definition we have $g_j(x,y)=\lambda_j$ and $h_j(x,y)=\lambda_{m+j}$ for $j=1,\dots,m$. Inserting this into (\ref{unitcircle}) we obtain
\begin{align}\label{lambdas}
\lambda_j^2+\lambda_{m+j}^2=1.
\end{align}
Hence for any $j=1,\dots,m$ there exists a unique number $\varphi_j\in [0,\pi)$ such that
\begin{align}
\lambda_j=\cos 2\varphi_j~~\textrm{and}~~\lambda_{m+j}=\sin 2\varphi_j.
\end{align}

We abbreviate $\varphi:=(\varphi_1,\dots,\varphi_m)\in [0,\pi)^m$ and write  $\cI_\varphi:=\cI(\lambda)$.

The following proposition gives  a convenient parametrization of the non-empty fibre sets.

\begin{prop}\label{charma} Let $\lambda \in \dR^{2m}$. Suppose that  the fibre set $\cZ(\cI_\varphi)\equiv\cZ(\cI(\lambda))$ is not empty. Then
\begin{align}
\cZ(\cI_\varphi)=\{ (x,y)\in \widehat{B}: a_j(x,y)\, \sin \varphi_j=b_j(x,y)\, \cos \varphi_j,\, j=1,\dots,m\, \}.\label{zvarphi}
\end{align}
\end{prop}
\begin{proof} Let  $ (x,y)\in \widehat{B} $ and $z=x+\ii \, y$. For $j=1,\dots,n$, we compute
\begin{align}\label{quofj}
\frac{f_j(z)}{\ov{f_j(z)}}&=\frac{a_j(x,y)+\ii \, b_j(x,y)}{a_j(x,y)-\ii \,b_j(x,y)}= \frac{a_j(x,y)^2-b_j(x,y)^2}{a_j(x,y)^2+b_j(x,y)^2}+\ii\,\frac{2a_j(x,y)b_j(x,y)}{a_j(x,y)^2+b_j(x,y)^2}\nonumber\\ &=g_j(x,y)+\ii \, h_j(x,y).
\end{align}
Set  $M_j:=\{w=f_j(z): z=x+\ii\, y, (x,y)\in \widehat{B}\, \}$ and $\theta_j=\lambda_j+\ii\, \lambda_{m+j}$.
Comparing (\ref{zlambda}),   (\ref{quofj}) and  (\ref{mtheta}) we conclude that
$\cZ(\cI_\varphi)$ is the intersection of fibre sets $ (M_j)_{\theta_j}$, $j=1,\dots,m$. Therefore, Lemma \ref{prelim} gives (\ref{zvarphi}).
\end{proof}

Note that   equation (\ref{zvarphi})  can be rewritten as
\begin{align}
\cZ(\cI_\varphi)=\{(x,y)\in \widehat{B}: a_j(x,y) =b_j(x,y)  \cot \varphi_j,\, j=1,\dots,m\, \},
\end{align}
with the interpretation that ``$b_j(x,y)=0$"\, if\, $ \varphi_j=0$.

Another useful   parametrization of the non-empty fibre sets is obtained by going back to the complex form. We set
\begin{align*}
\vartheta_j:=\exp (\ii\, \varphi_j),  j=1,\dots,m, ~\vartheta:=(\vartheta_1,\dots,\vartheta_m)~~\textrm{and}~~\cZ_\vartheta:=\cZ(\cI_\varphi).
\end{align*} and rename $\cZ(\cI_\varphi)$ by $\cZ_\vartheta$. Then, since the equation
$a_j(x,y)\sin \varphi_j=b_j(x,y)\cos \varphi_j$ is equivalent to
$f_j(z)\vartheta+\ov{f_j(z)}~ \ov{\vartheta}=0$, the non-empty fibre sets are given by
\begin{align}\label{fribrecomplex}
\cZ_\vartheta=\{(x,y)\in \widehat{B}:~ z=x+\ii\, y, ~f_j(z)\vartheta_j+\ov{f_j(z)}~ \ov{\vartheta_j}=0,~ j=1,\dots,m\, \}.
\end{align}
Here $\vartheta$ is an $m$-tuple of complex numbers $\vartheta_j$ which can be chosen of modulus one and satisfying $\arg \theta_j\subseteq [0,\pi)$.

\begin{thm}\label{fibreth}
Suppose that for each  $\varphi\in [0,\pi)^m$ such that the set $\cZ(\cI_\varphi)$ is not empty the quotient algebra $B/ \widehat{\cI}_\varphi$ has the moment property. Then $B$ and $A$ have  the moment property.
\end{thm}
\begin{proof}
By  Proposition \ref{charma} all fibres with non-empty fibre sets are of the form $\cI_\varphi$ for some $\varphi\in [0,\pi)^m$. Using this result  it follows from Proposition \ref{fibreth1} that the real algebra $B$ has the moment property. As discussed in Secrtion \ref{prel} this implies that the complex $*$- algebra $A$ has the moment property as well.
\end{proof}
To apply this theorem it  is useful to know some algebras which have the moment property. This is true if the   algebra has a  single generator (by \cite[Corollary 13.17]{sch17}) or if each fibre set is contained in a straight line (by Proposition \ref{mplines}).
\begin{rem}\label{fraction}
Recall from (\ref{hermitgen}) that the real algebra $B=A_\her$ is generated by $x_j,y_j$ and the fractions $\frac{1}{a_j^2+b_j^2}$,  $j=1,\dots,m$. From (\ref{zvarphi}) it follows  that in the quotient algebra  $B/ \widehat{\cI}_\varphi$,  $a_j$ is a constant multiple of $b_j$ or $b_j$ is a constant multiple of $a_j$. Hence the algebra $B/ \widehat{\cI}_\varphi$  is generated by $x_j,y_j$ and  $\frac{1}{a_j^2}$ {\it  or} $\frac{1}{b_j^2}$, respectively, for $j=1,\dots,m$.
That is,   the denominators of elements of $B/ \widehat{\cI}_\varphi$ are products of squares $a_j^2 $ or $b_j^2$, respectively.

In particular, if all complex polynomials $f_j(z)$ are linear, so are all real polynomials $a_j(x,y)$ and $b_j(x,y)$. Therefore, in this case all  elements of $B/ \widehat{\cI}_\varphi$  are fractions with even powers of {\it linear} polynomials $a_j$ or $b_j$ as denominators.
\end{rem}

\section{Linear Polynomials as Denominators}\label{genbisgaard}

The following example settles the simplest case.
\begin{exm}\label{bisgaard} {\it Case $d=1$}\\
In this example we suppose that $d=1$. Let $u$ be a zero of some polynomial $f_j$, where $j\in \{1,\dots,m\}$.  Upon multiplying $f_j(z)^{-1}\in A$ by the polynomial $f_j(z)(z-u)^{-1}$ we see that $(z-u)^{-1}\in A$. Conversely, by writing $f_j$ as a product of linear factors it follows that $f_j^{-1}$ belongs to an algebra if the inverses $(z-u)^{-1}$ of all zeros $u$ of $f_j(z)$ do. Thus the $*$-algebra $A$ is generated by $\dC[z,\ov{z}]$ and the inverses $(z-u)^{-1}$ for all zeros $u$ of some $f_j(z)$. That is,  we can assume without loss of generality  that $f_j(z)=z-u_j$ for some $u_j\in \dC$ and $j=1,\dots,m$.

Let $u_j=a_j+\ii\, b_j$, where $a_j,b_j \in \dR$. Then $a_j(x,y)=x-a_j, b_j(x,y)=y-b_j$. By Proposition \ref{charma}, each non-empty fibre set $\cZ(\cI_\varphi)$ is of the form
\begin{align*}
\cZ(\cI_\varphi)=\{ (x,y)\in \dR^{2}: (x-a_j)\sin\varphi_j=(y-b_j)\cos \varphi_j, ~ j=1,\dots,m\}.
\end{align*}
Such a set is obviously a subset of some line. Therefore, $B/ \hat{\cI}_\varphi$ has the moment property by Proposition \ref{mplines}. From the fibre theorem (Theorem \ref{fibreth}) it follows that the real algebra $B$, hence the complex $*$-algebra  $A$,\,has the moment property.

Thus, in the case $d=1$, $A$ has always the moment property.  The special case $m=1, f_1(z)=z$ is just Bisgaard's theorem.
\end{exm}

In the rest of this section we assume that all complex polynomials $f_j(z)$ are {\bf linear} in the variables $z_1,\dots,z_d$. Then $a_j(x,y)$ and $b_j(x,y)$ are  linear real polynomials of $x_1,\dots,x_d,y_1,\dots,y_d$ and each equation $a_j(x,y)\, \sin \varphi_j=b_j(x,y)\cos \varphi_j$ defines a hyperplane in $\dR^{2d}.$ Hence all non-empty fibre sets $\cZ(\cI_\varphi)$, defined by  (\ref{zvarphi}), are affine subspaces of $\dR^{2d}$.  Therefore, by Proposition \ref{mplines} and Theorem \ref{fibreth},  if each fibre set $\cZ(\cI_\varphi)$ is contained in a line, then the algebras $B$ and $A$ have the moment property. The same conclusion holds if each non-empty fibre set $\cZ_\vartheta$, defined by (\ref{fribrecomplex}), is contained in a  {\it real line} in $\dC^d$.

\begin{exm}\label{exampled2m4}
Let $d=2$ and $m=4$.  We make
following choice of linear polynomials $f_j=f_j(z_1,z_2)$:
\begin{align*}f_1=z_1+u_1,\ f_2=z_1+u_2,\ f_3=z_1+u_3,\ f_4=z_2,
\end{align*}
where $u_1,u_2,u_3$ are any three complex numbers that do not lie on
a real line in $\dC$.

Let
$\vartheta=(\vartheta_1,\vartheta_2,\vartheta_3,\vartheta_4)$ be a $4$-tuple of complex numbers of modulus one. Then  \begin{align*}
\{ z_1\in \dC: ~f_j(z_1)\vartheta_j+\ov{f_j(z_1)}~ \ov{\vartheta_j}=0,~ j=1,2,3\, \}.
\end{align*}consists of at most one point, since otherwise it would have to be a
line through $u_1,u_2$ and $u_3$. Therefore, the set \begin{align*}
\{( z_1,z_2)\in \dC^2: ~f_j(z_1,z_2)\vartheta_j+\ov{f_j(z_1,z_2)}~ \ov{\vartheta_j}=0,~ j=1,2,3,4\, \},
\end{align*}
 is either empty or a real line in $\dC^2$. Since this set is contained in the fibre set $\cZ_\vartheta$, it follows from the preceding discussion that  $B$ and $A$ have the moment property.
\end{exm}
Exactly the same argument as used in Example \ref{exampled2m4} and combined with the preceding discussion yields the following theorem.
 It may be viewed as one possible generalization of Bisgaard's theorem to higher dimensions.

\begin{thm}\label{mainthmp}
Suppose that $d\geq 2$ and set $m:=3d-2$.  Consider the   polynomials
\begin{align*}
f_{3j-2}(z)=z_j+u_j,~ f_{3j-1}(z)&=z_j+v_j,~ f_{3j}(z)= z_j+w_j\quad(j=1,\dots,d-1),\\~~ &f_{3d-2}(z)=z_d
\end{align*}
in $\dC[z_1,\dots,z_d]$,  wheres $u_j,v_j,w_j$ are
complex numbers such that $u_j,v_j,w_j$ do not lie on a real line in $\dC$ for each $j=1,\dots,d-1$. Then the complex $*$-algebra
\begin{align*}
A:=\dC[z_1,\dots,z_d,\ov{z_1},\dots,\ov{z_d},f_1^{-1},\dots,f_{m}^{-1},(\, \ov{f_1}\,)^{-1},\dots,(\, \ov{f_{m}}\, )^{-1} ],
\end{align*}
has the moment property.
\end{thm}
In the preceding theorem we have shown that there exist $3d-2$
linear polynomials $f_j(z)$  in $\dC[z_1,\dots,z_d]$ such that
the corresponding algebra $A$ has the moment property. The question
arises naturally whether the moment property can also be obtained
with  less than $3d-2$ polynomials that are linear, or even just
irreducible.\smallskip

\textbf{Question:}
\emph{Given $d$, do there exist $m<3d-2$ irreducible polynomials
$f_1,\dots,f_m$ in $\dC[z_1,\dots,z_d]$ such that the complex
$*$-algebra $A$ defined by (1) has the moment property?}
\smallskip

We are going to show (Theorem \ref{mpfails} below) that the answer is
no for linear polynomials and $d=2$. For this we need a few
technical preparations.

Let $x=(x_1,\dots,x_n)$ in the following. For an irreducible
polynomial $p$ in $\dR[x]$, there are several equivalent ways to
express that the hypersurface $p=0$ has Zariski dense $\dR$-points:

\begin{prop}
Let $p\in\dR[x]$ be irreducible. The following properties are
equivalent:
\begin{itemize}
\item[\rm(i)]
The $\dR$-points of the affine hypersurface
$V(p)=\{z\in\dC^n\colon p(z)=0\}$ are Zariski dense in $V(p)$;
\item[\rm(ii)]
the real algebraic set $\{x\in\dR^n\colon p(x)=0\}$ has dimension
$n-1$;
\item[\rm(iii)]
there exist $u,v\in\dR^n$ with $p(u)<0<p(v)$;
\item[\rm(iv)]
the field of fractions of the ring $\dR[x]/(p)$ can be ordered.
\end{itemize}
If these properties are satisfied for $p$, we'll say that the
irreducible polynomial $p$ is real.
\end{prop}

\begin{proof}
(i) and (ii) are equivalent by the definition of dimension of
real (semi-) algebraic sets \cite[Proposition 2.8.2]{bcr}.
Condition (iv) says that $-1$ is not a
sum of squares in the field of fractions of $\dR[x]/(p)$.
Therefore, the equivalence of (ii), (iii) and (iv) is contained in
\cite[Theorem 4.5.1]{bcr}.
\end{proof}

\begin{prop}\label{notmomentproperty}
Let $h\in\dR[x]$ be a product of irreducible polynomials, each of
which is real.
\begin{itemize}
\item[\rm(a)]
The sos cone $\Sigma A^2$ of the
real algebra $A=\dR[x,h^{-1}]$ is closed in the finest locally
convex topology of $A$.
\item[\rm(b)]
If $n\ge2$ then $A$ does not have the moment property.
\end{itemize}
\end{prop}

\begin{proof}
Let $g_1,\dots,g_r\in\dR[x]$ such that $h$ divides
$g=\sum_{i=1}^rg_i^2$, and let $p$ be an irreducible factor of $h$.
Then $p$ divides each single $g_i$ since the field of fractions of
$\dR[x]/(p)$ is real. Writing $g_i=pq_i$ with polynomials $q_i$, we
conclude that $g/p^2=\sum_{i=1}^rq_i^2$ is again a sum of squares of
polynomials. Iterating this argument shows that $g/h^2$ is a sum of
squares of polynomials.

(a): Fix integers $m,\,r\ge0$ and consider the linear subspace
\begin{align}\label{defU}U\>=\>\{ph^{-2m}\colon p\in\dR[x],\ \deg(p)\le r\}
\end{align}
of $A$.
Let $d=\deg(h)$. The map
$\phi\colon U\to\dR[x]_{\le(r+2md)}$, $f\mapsto fh^{2m}$ is linear.
Clearly, an element $f=ph^{-2m}\in U$ is sos in $A$ if and only if
(the polynomial) $p=\phi(f)$ is sos in $A$. By the initial remark,
the latter property is equivalent to $p=\phi(f)$ being a sum of
squares of polynomials. This means that $U\cap\Sigma A^2$ is the
preimage of the (truncated) sos cone in $\dR[x]_{\le(r+2md)}$ under
the linear map $\phi$. Since the sos cone $\Sigma\dR[x]^2$ in
$\dR[x]$ is closed,
this implies that $U\cap\Sigma A^2$ is closed in $U$. Since each
finite-dimensional subspace of $A$ is contained in a space $U$
as in (\ref{defU})  for some $m,r$, it follows that $\Sigma A^2$
is closed in the finest locally convex topology (by
\cite[Proposition A.28]{sch17} or \cite[Section 3.6]{marshall1}).

(b): If $n\ge2$,
there exist psd polynomials $p\in\dR[x]$ that are not sums of
squares of polynomials, for example the Motzkin polynomial.  By
the above remark, such $p$ is not a sum of squares in $A$ either.
Thus, $p\in A_+$ and $p\notin \Sigma A^2$. Since $\Sigma A^2$ is
closed by (i), $A_+$ is not the  closure of $\Sigma A^2$. Hence
$A$ does not have the moment property \cite[Proposition 13.9(ii)]{sch17}.
\end{proof}

\begin{thm} \label{mpfails}
Suppose that $d=2$ and  $m=3$. Let  $f_1,f_2,f_3\in\dC[z_1,z_2]$
be non-constant linear polynomials. Then the corresponding algebras $B$ and $A$ do not have the moment property.
\end{thm}
\begin{proof} First we recall that $B$ satisfies the moment property if and only $A$ does, as  discussed in the introduction.
Write $f_j(z)=g_j(z)+u_j$ with $0\ne g_j\in\dC[z_1,z_2]$
linear homogenous and $u_j\in\dC$.

First we treat the simple case when the linear forms $g_1,g_2,g_3$
are complex scalar multiples of each other. Then these forms are multiples of one linear combination of $z_1,z_2$, say $z_1$, without loss of generality. Hence $A$ is the tensor product of functions in $z_1,\ov{z_1}$ and the algebra $\dC[z_2,\ov{z_2}]$. Since $\dC[z_2,\ov{z_2}]$ does not  have the moment property (see e.g. \cite[Proposition 15.1]{sch17}),  neither does $A$.

From now on we suppose that $g_1,g_2,g_3$
are not complex multiples of each other. Then
 we may assume without loss of generality that
$g_1,g_2$ are linearly independent. Upon a suitable affine-linear
change of variables, we can assume that $f_1=z_1$ and $f_2=z_2$. Then $f_3$ is of the form
$f_3=uz_1+vz_2+w$ with $u,v,w\in\dC$. Since $f_3$ is not constant by assumption, $u\neq 0$ or $v\neq 0$. Without loss of generality we  assume that $u\neq0$; otherwise we interchange the variables. Multiplying $f_3$ by some complex number of modulus one, we can have that $w=(1+\ii )t$ for some $t\in \dR$. Further, scaling $z_1$ and $z_2$
suitably, we can  assume that $f_3$ is either $z_1+(1+\ii )t$ or
$z_1+z_2+(1+\ii )t$. These  changes replace $f_1, f_2 ,f_3$ by certain  complex multiples of $f_1,f_2,f_3$, respectively. Since these transformations do not change the corresponding algebra $A$, we can assume without loss of generality that $f_1,f_2,f_3$ are of this special  form.

Thus we have
\begin{align*}f_1(z)=z_1=x_1+\ii y_1,\,  f_2(z)=z_2=x_2+\ii y_2
\end{align*}
and concerning $f_3$ there are the following two cases:
\begin{align*}
\textrm{Case}~ 1:~~ & f_3(z)=z_1 +(1+\ii) t=(x_1+t)+\ii \, (y_1+t),\\
\textrm{Case}~ 2:~~& f_3(z)=z_1+z_2+(1+\ii) t=
(x_1+x_2+t)+\ii\, (y_1+y_2+t).
\end{align*}

To describe the fibre sets we abbreviate $s_j=\sin \varphi_j, c_j=\cos\varphi_j$ for $j=1,2,3$.
Then, by (\ref{zvarphi}),  the corresponding (non-empty) fibre sets  $\cZ(\cI_\varphi) $ are of the form \begin{align}\label{zlambda1}
\cZ(\cI_\varphi)=\{(x,y)\in \widehat{B}: s_1x_1=c_1y_1, s_2x_2=c_2y_2,s_3(x_1+t)=c_3 (y_1+t)\, \}.
 \end{align}
in Case 1 and
\begin{align}\label{zlambda2}
\cZ(\cI_\varphi)=\{(x,y)\in \widehat{B}: s_1x_1=c_1y_1, s_2x_2=c_2y_2,s_3(x_1+x_2+t)=c_3 (y_1+y_2+ t)\, \}.
 \end{align}
in Case 2.

 We set $\varphi_j:=\frac{\pi}{4}$ for $j=1,2,3,$ so that $s_j=c_j$. Then $x_1=y_1, x_2=y_2$  by the first two equation in (\ref{zlambda1})  and (\ref{zlambda2}), respectively, so the corresponding third equations in (\ref{zlambda1}) and (\ref{zlambda2}) are automatically fulfilled. That is,  the three equations in (\ref{zlambda1}) and in  (\ref{zlambda2}), respectively, are equivalent to the two equations $x_1=y_1$, $x_2=y_2$.

 Therefore, using Remark \ref{fraction}
we conclude that the fibre algebra $B/\widehat{\cI}_\varphi$ is  the algebra of rational functions on $\dR^2$ generated by $\dR[x_1,x_2]$ and the inverses $x_1^{-2}, x_2^{-2}$ and $(x_1+t)^{-2}$ in Case 1 and $(x_1+x_2+t)^{-2}$ in Case 2. Thus Proposition \ref{notmomentproperty} applies and it follows that this fibre algebra  $B/\widehat{\cI}_\varphi$ does not have the moment property. Then, by the implicaton (i)$\to$(ii) in Theorem \ref{fibreth}, $B$ and hence $A$ do not have the moment property.
 \end{proof}

\section{Positive semidefinite elements and sums of squares}\label{psdsos}
In this section we develop some results from real algebra. For three interesting real algebras it is shown that all positive semidefinite elements are sums of squares.  In particular, this implies that these algebras
satisfy the moment property by Havilands theorem. These three algebras will appear as  hermitean parts of  three semigroup $*$-algebras in the next section, see formulas (\ref{b1})--(\ref{b3}) below.

Let $A$ be a unital real algebra that is finitely generated. Recall that $A_+$ is the set of positive semidefinite, short psd, elements of $A$.
We say that $A$ satisfies $\rm psd=sos$ if every psd element in $A$
is a sum of squares of elements in $A$, or equivalently, if $\Sigma A^2=A_+$.

\begin{thm}\label{bisg}
The real algebra $A=\dR[x,y,\frac1{x^2+y^2}]$ satisfies $\rm psd=sos$:
Every element in $A$ that is non-negative on $\dR^2$ minus the origin
is a sum of squares of elements of~$A$.
\end{thm}

\begin{proof}
The ring $A$ consists of all fractions $\frac{p(x,y)}{(x^2+y^2)^m}$
where $p\in\dR[x,y]$ and $m\ge0$. To prove the assertion we
have to show: For every psd polynomial $p(x,y)$ in $\dR[x,y]$ there
is an integer $m\ge0$ such that the polynomial $(x^2+y^2)^mp(x,y)$ is
a sum of squares in $\dR[x,y]$. We claim that this assertion is
equivalent to the statement of \cite{ksv} Theorem 5.2, according to
which there exists $m\ge0$ such that $(1+x^2)^mp(x,y)$ is sos.
Indeed, both statements are equivalent to their common homogenized
version: For every homogeneous psd polynomial $q(x,y,z)$ in three
variables, there exists $m\ge0$ such that the polynomial
$(x^2+y^2)^mq(x,y,z)$ is a sum of squares in $\dR[x,y,z]$.
\end{proof}

The proof, as given in \cite{ksv}, rests on the cylinder theorem:

\begin{thm}\label{cylthm}
The real algebra $\dR[x,y,z]/(1-x^2-y^2)$ satisfies $\rm psd=sos$.
\end{thm}

Theorem \ref{cylthm} says that every non-negative polynomial on the
cylinder surface in $\dR^3$ is a sum of squares of polynomials on
this surface. This result can be seen as a strengthening of
Marshall's celebrated \emph{strip theorem} \cite{ma}. It is proved
in \cite{sw} Theorem~2.

\begin{thm}\label{bisgbdd}
The real algebra $B=\dR[x,y,\frac{x^2}{x^2+y^2},\frac{xy}{x^2+y^2}]$
satisfies $\rm psd=sos$.
\end{thm}

\begin{proof}
We'll derive Theorem \ref{bisgbdd} from Theorem \ref{bisg}.
The ring $A=\dR[x,y,(x^2+y^2)^{-1}]$ from Theorem \ref{bisg} can be
given a $\dZ$-grading, namely $A=\bigoplus_{d\in\dZ}A_d$ with
$$A_d\>=\>\Bigl\{\frac{p(x,y)}{(x^2+y^2)^n}\colon p(x,y)\in\dR[x,y]
\text{ homogeneous of degree }2n+d\Bigr\}$$
The ring $B$ is a subring of $A$. Since $B$ is generated by
homogeneous elements of $A$ that have non-negative degrees, the
inclusion $B\subseteq\bigoplus_{d\ge0}A_d$ is obvious. We first show
that equality holds here, and start with degree $d=0$. So let
$f\in A_0$, say $f=p(x,y)\cdot(x^2+y^2)^{-n}$ where $n\ge0$ and
$p(x,y)\in\dR[x,y]$ is homogeneous of degree~$2n$. There exists a
ternary form $q(u,v,w)$ of degree $n$ in $\dR[u,v,w]$ such that
$p(x,y)=q(x^2,xy,y^2)$. Therefore
$$\frac{p(x,y)}{(x^2+y^2)^n}\>=\>q\Bigl(\frac{x^2}{x^2+y^2},\,
\frac{xy}{x^2+y^2},\,1-\frac{x^2}{x^2+y^2}\Bigr)$$
holds, showing that $f\in B$.

When $d\ge0$ is arbitrary, every element $f\in A_d$ can be written
as a finite sum $f=\sum_iq_if_i$, where the $q_i$ are forms of
degree~$d$ in $\dR[x,y]$ and the $f_i$ lie in $A_0$. So $q_i\in B$,
and also $f_i\in B$ by what was just shown. We conclude $f\in B$, and
so $B=\bigoplus_{d\ge0}A_d$ has been shown.

Now let $f=f(x,y)\in B$ be a psd element of $B$. Then the rational
function $f(x,y)$ is non-negative on $\dR^2\setminus\{(0,0)\}$. From
Theorem \ref{bisg} we know that $f$ is a sum of squares in $A$, so
there is an identity $f=\sum_if_i^2$ with $f_i\in A$. We claim that
each $f_i$ lies in $B$. Assuming that this is false, let $m<0$ be the
smallest degree for which some $f_i$ has a non-zero homogeneous
component in degree~$m$. Then $f=\sum_if_i^2$ has a non-zero
homogeneous component in degree $2m<0$, a contradiction.
So indeed $f_i\in B$ for every index~$i$, and so $f$ is a sum of
squares in $B$.
\end{proof}

We remark that
the degree zero component $B_0$ of $B$ is the algebra
$B_0=A_0=\dR[\frac{x^2}{x^2+y^2},\,\frac{xy}{x^2+y^2}]$. It is easy to
see that $B_0=A_0$ agrees with the ring of all elements of $A$ that
are bounded (as functions on $\dR^2\setminus\{(0,0)\}$).

\section{Applications to three semigroup $*$-algebras}\label{semigroupappl}

In this section we consider the following three $*$-semigroups:\smallskip

1.) $\dZ^2$ with involution $(k,n)^*=(n,k)$ for $k,n\in \dZ$,

2.)  $\dN_0\times \dZ$ with involution $(k,n)^*=(k,-n)$ for $k\in \dN_0$,  $n\in \dZ$.

3.) ${\sN}_+:=\{(k,n)\in \dZ^2:k+n\geq 0\}$ with involution $ (k,n)^*=(n,k)$ for $k,n\in \dZ.$\smallskip

In all three cases the semigroup composition is coordinatewise addition.

The semigroup $*$-algebras $ \dC[S]$ of all three $*$-semigroups have the moment property, that is, each  positive linear functional $L$ on $\dC[S]$  is a moment functional. In the case $S=\dZ^2$ this  is T. M. Bisgaard's theorem
 \cite{bisgaard}, as discussed above. For $S=\dN_0\times\dZ$ this result is A. Devinatz' theorem \cite{devinatz} and for $S={\sN_+}$ it is due to J. Stochel and F. H. Szafraniec \cite{sz}. In \cite{sch17} all three results are derived from the fibre theorem, see Theorem 15.15, Proposition 15.15 and Theorem 15.14, respectively. The original approaches to the  latter two theorems emphasize the closed interplay between moment problems and Hilbert space operators. Devinatz'  original proof uses  the spectral theorem, while the approach of  Stochel and Szafraniec is  based on the polar decomposition and provides a characterization of subnormality of some unbounded operators (see also \cite{sz89}).

Let $S$ be an arbitrary  finitely generated $*$-semigroup.  If
$\Sigma_h\dC[S]^2=\dC[S]_+$,  then it follows from Haviland's
theorem (Proposition \ref{haviland}) that the complex $*$-algebra
$\dC[S]$ has the moment property. The converse is not true. For
instance, if $S=\dN_0\times \dZ^d$ and $d\geq 2$, then $\dC[S]$
has also the moment property (see \cite[Proposition 15.5]{sch17}),
but $\Sigma_h\dC[S]^2\neq \dC[S]_+$. (If $d\geq 3$ it suffices to note there exist trigonometric polynomials  which are nonnegative on the torus $\dT^d$, but not hermitean sums of squares.)

 The following theorem shows that the  three $*$-semigroups mentioned above satisfy the stronger property $\sum \dC[S]^2=\dC[S]_+$.

 \begin{thm}\label{semigroup}
 Let $\dC[S]$ be one of the three $*$-semigroups $\dZ^2$, $\dN_0\times \dZ$,
 $\sN_+$ defined above. Then $\Sigma_h\dC[S]^2=\dC[S]_+$.
 \end{thm}
\begin{proof}
The strategy of proof is the following. First we define a $*$-isomorphism of $\dC[S]$ to a certain complex $*$-algebra $A$ of rational functions. The hermitean part $B:=A_\her$ of $A$ is a unital real algebra. Then we apply  results of Section \ref{psdsos} to the corresponding real algebra $B$ to conclude that $\Sigma B^2=B_+$. Recall  that $\Sigma_h A^2=\Sigma B^2$ and the map $\chi\to \chi\lceil  B$ is a bijection of $\hat{A}$ on $\hat{B}$ which maps $A_+$ on $B_+$. Therefore, the equality  $\Sigma B^2=B_+$ implies that $\Sigma_h A^2 =A_+$, which is the assertion.

Thus, it remains to determine   the algebra $B$ and to prove that $\Sigma B^2=B_+.$ \smallskip

1.) $S=\dZ^2$.

As noted  in the introduction, the map $(k,n)\to z^k\ov{z}^n$ gives a $*$-isomorphism of $\dC[S]$  to $A=\dC[z,\ov{z}, z^{-1},\ov{z}^{-1}]$. We write $z=x+\ii y$, with $x={\rm Re } , z, y={\rm Im}\, z$. Since $z^{-1}=\frac{x-\ii y}{x^2+y^2}$,  a short computation shows  that $A$ is generated by the four real-valued functions
$x,y, \frac{x}{x^2+y^2}\frac{y}{x^2+y^2}$ and hence by $x,y, \frac{1}{x^2+y^2}$, so that
\begin{align}\label{b1}B=\dR\Big[x,y, \frac{1}{x^2+y^2}\Big].
\end{align}
Clearly, $\hat{B}=\dR^2\backslash \{(0,0)\}$. The equality $\Sigma B^2=B_+$ is just  Theorem \ref{bisg}.\smallskip

2.) $S=\dN_0\times \dZ.$\smallskip

Then the complex $*$-algebra
$\dC[S]$  is $*$-isomorphic to the complex  polynomials on the cylinder surface $C:=\{(x,y,z)\in \dR^3: x^2+y^2=1\}$. From this  fact we obtain .
\begin{align}\label{b2}
B= \dR[x,y,z]/(1- x^2-y^2).
\end{align}
The character set $\hat{B}$ consists of the point evaluations at points of $C$. Theorem \ref{cylthm} yields the equality $\Sigma B^2=B_+$.
\smallskip

3.) $S={\sN}_+$.
\smallskip

The map $(k,n)\mapsto z^k\ov{z}^n$\, defines a $*$-isomorphism of  $\dC[\sN_+]$ on the $*$-algebra\, $A$ spanned by the functions $z^k\ov{z}^n$ on $\dC$, where $k+n\geq 0,k,n\in \dZ$. 	Clearly,  as a $*$-algebra $A$ is generated by $z$ and $v(z):=z\, \ov{z}^{-1}$. Since
\begin{align*}
1+v(z)=2\frac{x^2+\ii xy}{x^2+y^2}~~\textrm{and}~~1-v(z)=2\frac{x^2-\ii xy}{x^2+y^2},
\end{align*} we conclude that the $*$-algebra $A$ is also generated by the four real-valued functions $x,y,  \frac{x^2}{x^2+y^2},\frac{xy}{x^2+y^2}$. This implies that
\begin{align}\label{b3}
B=\dR\Big[x,y, \frac{x^2}{x^2+y^2},\frac{xy}{x^2+y^2}\Big].
\end{align}
Clearly, $\hat{B}=\dR^2\backslash \{(0,0)\}$. From Theorem \ref{bisgbdd} we obtain $\Sigma B^2=B_+$.

This completes the proof of Theorem \ref{semigroup}.
\end{proof}

\section*{Acknowledgement}

This research was carried out during the  stay of the second author at the Zukunftskolleg of the University of Konstanz. He  would like to thank the Zukunftskolleg for his support.
\bibliographystyle{amsalpha}

\begin{thebibliography}{xxxx}

\bibitem[B89]{bisgaard}
Bisgaard, T.M.: The two-sided complex moment problem. Ark. Mat. \textbf{27}(1989), 23--28.
\bibitem[BCR98]{bcr}
Bochniak,  J. , M. Coste and M.-F. Roy: \emph{Real Algerbaic Geometry}, Springer-Verlag, Berlin, 1998.
\bibitem[D55]{devinatz}
Devinatz, A., Integral representations of positive definite functions II. Trans Amer. Math. Soc. \textbf{77}(1955), 455--480.

\bibitem[DE70]{duef}
Dubois, D.W. and  G.~Efroymson:
Algebraic theory of real varieties~I.
In: Studies and essays (presented to Yu-Why Chen on his 60th birthday),
pp~107--135, Math.\ Res.\ Center, Nat.\ Taiwan Univ., Taipei (1970).

\bibitem[M08]{marshall1}  Marshall, M.: \emph{Positive Polynomials}. Math. Surveys and Monographs, Amer. Math. Soc., RI, Providence, 2008.

\bibitem[M10]{ma}
Marshall, M.:
Polynomials non-negative on a strip.
Proc.\ Amer.\ Math.\ Soc.\ \textbf{138}, 1559--1567 (2010).

\bibitem[PD01]{PD}  Prestel, A. and C.N. Delzell: \emph{Positive Polynomials}. Monographs in Math., Springer-Verlag, Berlin, 2001.
\bibitem[KSV23]{ksv}
Klep, I.,  C.~Scheiderer,  and J.~Vol\v ci\v c:
Globally trace-positive noncommutative polynomials and the
unbounded tracial moment problem.
Math.\ Ann.\ \textbf{387}, 1403--1433 (2023).

\bibitem[SW17]{sw}
Scheiderer, C. and S.~Wenzel:
Polynomials nonnegative on the cylinder.
In: Ordered algebraic structures and related topics, Contemporary Math.
\textbf{697}, pp. 291--300, Amer. Math. Soc., Providence, R.I., 2017.

\bibitem[Sm03]{sch03}
Schm\"udgen, K.: On the moment problem of closed semi-algebraic sets. J. reine angew. Math. \textbf{558}(2003), 225--234.
\bibitem[Sm16]{sch16}
 Schm\"udgen, K.:  A fibre theorem for moments problems and some applications. Israel J. Math. \textbf{28}(2016), 43--66.

\bibitem[Sm17]{sch17} Schm\"udgen, K.: \emph{The Moment Problem}. Graduate Texts in Math. \textbf{277}, Springer,  Cham, 2017.

\bibitem[Sm20]{sch20} Schm\"udgen, K.: \emph{An Invitation to Unbounded Representations of $*$-Algebras on Hilbert Space}. Graduate Texts in Math. \textbf{285},   Springer, Cham, 2020.


\bibitem[StS89]{sz89}
Stochel, J. and F.H. Szafraniec: On normal extensions of unbounded normal operators, II.
Acta Sci. Math. (Szeged)   \textbf{53}(1989), 137--177.


\bibitem[StS98]{sz}
Stochel, J. and F.H. Szafraniec: The complex moment problem and subnormality: a polar decomposition approach.
J. Funct. Anal. \textbf{159}(1998), 432--491.




\end{thebibliography}


\end{document}